\documentclass[a4paper,12pt,reqno]{amsart}
\usepackage{amssymb}
\usepackage{amsmath, mathrsfs}
\usepackage{enumerate}

\newtheorem{theorem}{Theorem}[section]

\newtheorem{lem}[theorem]{Lemma}
\newtheorem{prop}[theorem]{Proposition}

\theoremstyle{definition}

\theoremstyle{remark}
\newtheorem{rem}[theorem]{Remark}

\numberwithin{equation}{section}

\newcommand{\rad}[1]{\textup{\rm rad}({#1})}
\newcommand{\rada}{{\textup{\rm rad}(A)}}
\newcommand{\prim}[1]{\textup{\rm Prim}({#1})}
\newcommand{\prima}{{\textup{\rm Prim}(A)}}
\newcommand{\CC}{{\mathbb C}}

\newcommand{\rest}[2]{{{#1}_{\kern-.5pt|{#2}}}}
\newcommand{\idI}{{\mathscr I}}
\def\C*{{\sl C*}-algebra}
\def\Cs*{{\sl C*}-subalgebra}

\begin{document}
\title[Spectral isometries on non-simple \textsl{C*}-algebras]{Spectral isometries on non-simple \textsl{C*}-algebras}

\author{Martin Mathieu}
\address{Department of Pure Mathematics, Queen's University Belfast, Bel\-fast BT7 1NN\\ Northern Ireland}
\email{m.m@qub.ac.uk}
\author{Ahmed R.\ Sourour}
\address{Department of Mathematics and Statistics, University of Victoria, Victoria, BC, Canada V8W 3R4}
\email{sourour@math.uvic.ca}

\thanks{This paper was written during a visit of the first-named author to the University of Victoria in May 2011
supported by a Research in Pairs grant of the London Mathematical Society.
The second-named author's research is supported by an NSERC discovery grant.}

%\date{\today}

\begin{abstract}
We prove that unital surjective spectral isometries on certain non-simple unital \C*s are Jordan isomorphisms.
Along the way, we establish several general facts in the setting of semisimple Banach algebras.
\end{abstract}

\maketitle

\section{Introduction}\label{sect:intro}

\noindent
Let $A$ and $B$ be unital \C*s. A linear mapping $T\colon A\to B$ preserving the spectral radius of every element,
that is, $r(Ta)=r(a)$ for every $a\in A$, is called a \textit{spectral isometry}. This is a seemingly very weak property,
and one wonders how much more of the structure of a \C* such a mapping may be able to preserve. It is known that,
if $A$ is commutative and $T$ is unital, that is, $T1_A=1_B$, and surjective, then $T$ is an algebra isomorphism
\cite[Proposition~2.2]{Mat95}. It is also known that, if $A$ is finite-dimensional and $T$ is a unital surjective spectral
isometry, then $T$ preserves squares, that is, $T(a^2)=(Ta)^2$ for every $a\in A$ \cite[Corollary~5]{MaSou}; see
also \cite[Corollary~1.4]{Co07}. The latter property of $T$ implies that $T$ preserves the Jordan structure, that is,
is a \textit{Jordan homomorphism\/}: $T(ab+ba)=(Ta)(Tb)+(Tb)(Ta)$ for all $a,b\in A$.
In the non-commutative setting, this is the most one can hope for.

Examples of spectral isometries are spectrum-preserving mappings---$\sigma_B(Ta)=\sigma_A(a)$ for every $a\in A$;
these have been studied intensively over the past decades. Aupetit showed in \cite{Aup00} that every surjective spectrum-preserving
linear mapping between von Neumann algebras is a bijective Jordan homomorphism (i.e., a \textit{Jordan isomorphism}).
It is not known (yet) whether this extends to general unital \C*s. Nevertheless, in~\cite{MaSc}, the first-named author
conjectured that this conclusion should hold for all unital surjective spectral isometries.
Many results confirming this conjecture under additional hypotheses have appeared over the last few years,
see \cite{BeBouSa08}, \cite{Co09}, \cite{CoRe}, \cite{Mat95}, \cite{Mat09},  \cite{MMRud} and \cite{MaSou},  for example.
Sometimes, the conclusion that $T$ preserves the Jordan product is even valid under the weaker assumption that $T$
is \textit{spectrally bounded}, i.e., $r(Ta)\leq M\,r(a)$ for all $a\in A$ and some constant $M>0$, cf.~\cite{LinM}, \cite{MM04b}, \cite{MaSc2}.
However, in the presence of traces, one cannot expect such general statement; for a detailed discussion of this, see
Section~5 in~\cite{Mat09}.

Some of the methods to obtain the conclusion that a spectral isometry is a Jordan isomorphism are global in nature;
see, for instance, \cite{Co09}, \cite{LinM} or~\cite{MaSc2}. Others are reduction techniques using suitable quotients
of the algebras in question. In this paper, we shall follow this second approach by combining ideas from~\cite{CoRe},
\cite{Mat09}, \cite{MMRud} and~\cite{MaSou}. Our main result, Theorem~\ref{thm:main}, states that a unital spectral isometry
from a unital \C* with real rank zero and without tracial states that contains enough maximal ideals onto a unital \C* must be
a Jordan isomorphism. This is obtained by first working out a reduction to simple quotients which is valid in general semisimple
Banach algebras in Section~\ref{sect:pre} and then applying results known for simple \C*s, in particular~\cite{MM04b}.
The same approach works for separable unital \C*s with Hausdorff spectrum (Theorem~\ref{thm:typeI-hausdorff-sep}),
thus generalising results obtained in~\cite{MMRud}.

We conclude this introduction with a few comments on the general setting of our results.
Suppose $T\colon A\to B$ is a spectral isometry between two unital Banach algebras $A$ and~$B$.
The surjectivity assumption is inevitable, as is well known. If the domain of  $T$ is semisimple, then
$T$ is injective \cite[Proposition~4.2]{MaSc} and the inverse $T^{-1}$ is a bijective spectral isometry as well.
In this case, the codomain is semisimple too (Lemma~\ref{lem:rad-inv} below) and it follows that $T$ is bounded
\cite[Corollary~5.46]{All11}. The open mapping theorem entails that $T^{-1}$ is bounded as well.
When $T$ is non-unital and $A$ and $B$ are \C*s, then $T1$ is a central unitary;
thus, replacing $T$ by $\tilde T$ defined by $\tilde T(x)=(T1)^{-1}Tx$, $x\in A$ we can reduce the general to the unital case
(see the proof of \cite[Corollary~2.6]{LinM}).

\section{Preparations}\label{sect:pre}

\noindent
In this section we collect together several auxiliary results needed for the main results in the next section.
These are in fact valid in the framework of (semisimple) Banach algebras.
When $A$ is a Banach algebra, $\rada$ will denote its Jacobson radical and $Z(A)$ its centre.

The first lemma is easily deduced from known properties of spectral isometries, see, e.g.,~\cite{MaSc},
but we include a proof for completeness.
\begin{lem}\label{lem:rad-inv}
Let\/ $A$ and\/ $B$ be unital Banach algebras. Let\/ $T\colon A\to B$ be a surjective spectral isometry.
Then\/ $T\,\rada=\rad B$.
\end{lem}
\begin{proof}
Take $a\in\rada$ and $y\in B$ such that $r(y)=0$. Choose $x\in A$ with $y=Tx$; then, by hypothesis, $r(x)=r(y)=0$.
It follows that
\begin{equation*}
r(Ta+y)=r(T(a+x))=r(a+x)=0
\end{equation*}
so that $Ta\in\rad B$, by Zem\'anek's characterisation of the radical \cite[Theorem~5.40]{All11}.

Conversely, take $b\in\rad B$ and let $a\in A$ be such that $b=Ta$. Let $x\in A$ be quasinilpotent. Then
\begin{equation*}
r(a+x)=r(T(a+x))=r(b+Tx)=0
\end{equation*}
since $Tx$ is quasinilpotent. As before, it follows that $a\in\rada$ wherefore $b\in T\,\rada$.
We conclude that $T\,\rada=\rad B$.
\end{proof}
In particular, the image of a semisimple Banach algebra under a surjective spectral isometry  is semisimple.

The next result, too,  is an immediate consequence of known results together with Lemma~\ref{lem:rad-inv}.
Compare also with \cite[Proposition~2.2]{Mat95}.
\begin{lem}\label{lem:commutative}
Let\/ $T\colon A\to B$ be a unital surjective spectral isometry from the semisimple unital Banach algebra\/ $A$ onto the unital
Banach algebra\/ $B$. If\/ $A$ is commutative then\/ $B$ is commutative.
\end{lem}
\begin{proof}
By Lemma~\ref{lem:rad-inv}, $B$ is semisimple and by \cite[Corollary~4.4]{MaSc}, $B=TA=TZ(A)=Z(B)$ so that $B$ is commutative.
\end{proof}
The following observation allows us conveniently to pass to the quotient of a spectral isometry.
\begin{lem}\label{lem:quotient-spisom}
Let\/ $S\colon B\to B$ be a unital surjective spectral isometry on a unital Banach algebra~$B$.
Let\/ $I$ be a closed ideal of\/ $B$ such that the quotient algebra\/ $B/I$ is semisimple.
If\/ $r(Sy+I)=r(y+I)$ for all\/ $y\in B$ then\/ $S$ induces a unital surjective spectral isometry\/ $S^I\colon B/I\to B/I$ such that\/
$S^I(y+I)=Sy+I$ for all\/ $y\in B$.
\end{lem}
\begin{proof}
Clearly, all we need to show is that $SI\subseteq I$; in this case, $S^I(y+I)=Sy+I$, $y\in B$ defines a unital surjective mapping on $B/I$
which is a spectral isometry by assumption.

Let $x\in I$ and take $y\in B$. Choose $x'\in B$ such that $y=Sx'$. For each $\lambda\in\CC$, we have
\begin{equation*}
\begin{split}
r\bigl(\lambda(Sx+I)+y+I\bigl) &= r(\lambda Sx+y+I)\\
                                                     &=r(S(\lambda x+x')+I)\\
                                                     &=r(\lambda x+x'+I)=r(x'+I).
\end{split}
\end{equation*}
Applying Liouville's theorem \cite[Corollary~5.43]{All11} to the subharmonic function $\lambda\mapsto r\bigl(\lambda(Sx+I)+y+I\bigl)$
we find that
\[
r(Sx+I+y+I)=r(y+I)\qquad(y\in B).
\]
Zem\'anek's characterisation of the radical implies that $Sx+I\in\rad{B/I}=0$ which yields that $Sx\in I$.
\end{proof}
We now establish a key result for our main theorem; the method of proof is inspired by~\cite{CoRe}.
Here, and in the following, $X^c$ denotes the commutant
\[
X^c=\{y\in A\mid yx=xy\ \text{for all}\ x\in X\}
\]
of a subset $X\subseteq A$ in~$A$.
\begin{lem}\label{lem:modI}
Let\/ $A$ and\/ $B$ be unital semisimple Banach algebras, and let\/ $T\colon A\to B$ be a unital surjective
spectral isometry. Let\/ $I$ be a closed ideal of\/ $B$ such that\/ $B/I$ is semisimple and that each unital
surjective spectral isometry\/ $S\colon B/I\to B/I$ is multiplicative or anti-multiplicative. Let\/ $a\in A$ and put\/
$A_0=\{a\}^{cc}$. Let\/ $B_0=TA_0$. For all\/ $b_1,b_2\in B_0$ and\/ $x\in\{a\}^c$, we have
\begin{equation}\label{eq:commuting}
T^{-1}(b_1b_2)\,x+I=x\,T^{-1}(b_1b_2)+I.
\end{equation}
\end{lem}
\begin{proof}
Let $x\in\{a\}^c$ and $\lambda\in\CC$. Using the inverse spectral isometry $T^{-1}\colon B\to A$ we define
$S_{\lambda,x}\colon B\to B$ by $S_{\lambda,x}(y)=T\bigl(e^{-\lambda x}\,T^{-1}(y)\,e^{\lambda x}\bigr)$, $y\in B$.
Then $S_{\lambda,x}$ is a unital surjective spectral isometry of~$B$ with $S_{0,x}=\textup{id}_B$ and
$\rest{S_{\lambda,x}}{B_0}=\textup{id}_{B_0}$.
We have
\begin{equation*}
r(S_{\lambda,x}(y)+I) \leq r\bigl(T\bigl(e^{-\lambda x}\,T^{-1}(y)\,e^{\lambda x}\bigr)\bigr)=r(y)
\end{equation*}
independent of $x$ and~$\lambda$. Therefore the subharmonic function $\lambda\mapsto r(S_{\lambda,x}(y)+I)$
is constant by Liouville's theorem and, taking $\lambda=0$, we obtain
\begin{equation}
r(S_{\lambda,x}(y)+I)=r(y+I)\qquad(y\in B,\,x\in\{a\}^c,\,\lambda\in\CC).
\end{equation}
By Lemma~\ref{lem:quotient-spisom}, $S_{\lambda,x}I\subseteq I$ and we obtain an induced unital surjective
spectral isometry $S^I=S^I_{\lambda,x}$ on $B/I$.
By hypothesis, $S^I$ is either multiplicative or anti-multiplicative. Evidently, $S^I_{0,x}$ is the identity on~$B/I$.

Suppose $B/I$ is commutative. Then, by Nagasawa's theorem, $S^I$ is multiplicative. Hence, for the purpose of the remaining
argument, we can assume that $B/I$ is non-commutative. For fixed $x$, the set
$\{\lambda\in\CC\mid S^I_{\lambda,x}\text{ is multiplicative}\}$
contains~$0$ and is closed since $\lambda\mapsto S^I_{\lambda,x}$ is continuous.
Suppose that the above set is not open and take a point $\lambda_0$ in it which is an adherence point of the complement.
Using the continuity of $\lambda\mapsto S^I_{\lambda,x}$ again, we conclude that
\[
S^I_{\lambda_0,x}(b_1+I)\,S^I_{\lambda_0,x}(b_2+I)=S^I_{\lambda_0,x}(b_2+I)\,S^I_{\lambda_0,x}(b_1+I)\qquad(b_1,b_2\in B),
\]
as $S^I_{\lambda_0,x}$ is both multiplicative and anti-multiplicative. However, as $S^I_{\lambda_0,x}$ is surjective, it
follows that $B/I$ is commutative, a contradiction. We thus conclude that $S^I_{\lambda,x}$ is an algebra
isomorphism for all $\lambda\in\CC$ and every $x\in\{a\}^c$.

\smallskip
Let $b_1,b_2\in B_0$. As $S_{\lambda,x}(b_i)=b_i$, $i=1,2$, we obtain
\begin{equation*}
S^I_{\lambda,x}(b_1b_2+I)=S^I_{\lambda,x}(b_1+I)\,S^I_{\lambda,x}(b_2+I)
                                                            =(b_1+I)\,(b_2+I)
\end{equation*}
or, equivalently,
\begin{equation}\label{eq:mult-modI}
T\bigl(e^{-\lambda x}\,T^{-1}(b_1b_2)\,e^{\lambda x}\bigr)+I=(b_1+I)\,(b_2+I)\qquad(\lambda\in\CC).
\end{equation}
Taking derivatives in~\eqref{eq:mult-modI}, evaluating at $\lambda=0$ and using that $T$ is injective we find that
\begin{equation}
T^{-1}(b_1b_2)\,x-x\,T^{-1}(b_1b_2)+I=0
\end{equation}
as claimed.
\end{proof}

\section{Results}\label{sect:results}

\noindent
Our first result in this section explains the significance of commutative subalgebras for the general conjecture.
For an element $a$ in a unital \C* $A$, $C^*(a)$ stands for the unital \Cs* of $A$ generated by~$a$.
\begin{prop}\label{prop:comm-subalg}
Let\/ $A$ be a unital \C* and let\/ $B$ be a unital Banach algebra. The following conditions on a unital surjective
spectral isometry\/ $T\colon A\to B$ are equivalent.
\begin{enumerate}[\rm(a)]
\item $T$ is a Jordan isomorphism;
\smallskip
\item $TA_0$ is a subalgebra of\/ $B$ for every commutative unital subalgebra\/ $A_0$ of\/~$A$;
\smallskip
\item $T\bigl(\{a\}^{cc}\bigr)$ is a subalgebra of\/ $B$ for every element\/ $a\in A_{sa}$;
\smallskip
\item $T\,C^*(a)$ is a subalgebra of\/ $B$ for every element\/ $a\in A_{sa}$.
\end{enumerate}
\end{prop}
\begin{proof}
By Lemma~\ref{lem:rad-inv}, $B$ is automatically semisimple.

\noindent
(a) ${}\Rightarrow{}$ (b)\enspace It is well known that every Jordan isomorphism between semisimple algebras preserves commutativity
(see, e.g., \cite{JaRi}) and therefore restricts to a multiplicative mapping on each commutative subalgebra; as a result, $TA_0$ is
a subalgebra of\/ $B$ for every commutative unital subalgebra\/ $A_0$ of\/~$A$.

\noindent
(b) ${}\Rightarrow{}$ (c) and (b) ${}\Rightarrow{}$ (d)\enspace are trivial.

\noindent
We finally show that (c) and (d) individually imply~(a).
To this end, set $A_0=\{a\}^{cc}$ or $A_0=C^*(a)$, respectively, for an arbitrary fixed element $a\in A_{sa}$.
Note that $\{a\}^{cc}$ is a \Cs* by the Fuglede--Putnam theorem.
Since $T$ is a topological iso\-morph\-ism between $A$ and $B$, the subspace $B_0=TA_0$ is closed in~$B$.
By hypothesis, $B_0$ thus is a unital Banach algebra. Let $T_0\colon A_0\to B_0$ denote the restriction of $T$ to $A_0$ which is
a unital surjective spectral isometry. By Lemma~\ref{lem:rad-inv}, $B_0$ is semisimple and by Lemma~\ref{lem:commutative},
$B_0$ is commutative. Nagasawa's theorem \cite[Theorem~4.1.17]{Aup91} entails that $T_0$ is multiplicative wherefore
$T(a^2)=T_0(a^2)=(T_0a)^2=(Ta)^2$.
It follows that $T$ is a Jordan isomorphism.
\end{proof}
\begin{rem}\label{rem:commutative-general}
If, in the above proposition, $A$ is merely assumed to be a unital semisimple Banach algebra instead of a unital \C*,
the condition that $T\bigl(\{a\}^{cc}\bigr)$ is a subalgebra of\/ $B$ for every element\/ $a\in A$
still yields that $T$ is invertibility preserving. This can be seen by using Lemma~\ref{lem:rad-inv} to obtain
$T_0\,\rad{A_0}=\rad{B_0}$, with $A_0=\{a\}^{cc}$ and $B_0=TA_0$, and therefore a unital surjective spectral isometry
$T_1\colon A_1=A_0/\rad{A_0}\to B_1=B_0/\rad{B_0}$
which must be an isomorphism by Nagasawa's theorem. It follows that
\begin{equation*}
\sigma_A(a)=\sigma_{A_0}(a)=\sigma_{A_1}(a+\rad{A_0})=\sigma_{B_1}(T_0a+\rad{B_0})=\sigma_{B_0}(T_0a)
                         \supseteq\sigma_B(Ta).
\end{equation*}
\end{rem}
We will now combine Proposition~\ref{prop:comm-subalg} with the results in Section~\ref{sect:pre}.
\begin{prop}\label{prop:gen-criteria}
Let\/ $A$ and\/ $B$ be unital semisimple Banach algebras, and let\/ $T\colon A\to B$ be a unital surjective
spectral isometry. Suppose that\/ $B$ has a separating family\/ $\idI$ of closed ideals\/ $I$ such that\/ $B/I$ is semisimple
and that each unital surjective spectral isometry\/ $S^I\colon B/I\to B/I$ is multiplicative or anti-multiplicative.
Then\/ $T$ preserves invertibility. If, moreover, $A$ is a \C*, then\/ $T$ is a Jordan isomorphism.
\end{prop}
\begin{proof}
Let $a\in A$ and put $A_0=\{a\}^{cc}$. Let $B_0=TA_0$. Take $b_1,b_2\in B_0$ and $x\in\{a\}^c$.
By Lemma~\ref{lem:modI}, $T^{-1}(b_1b_2)\,x+I=x\,T^{-1}(b_1b_2)+I$ for every $I\in\idI$. As $\idI$ is separating, this yields
$T^{-1}(b_1b_2)\,x=x\,T^{-1}(b_1b_2)$ which entails that $T^{-1}(b_1b_2)\in A_0$, equivalently, $b_1b_2\in B_0$.
We conclude that $B_0$ is a (closed, unital) subalgebra of~$B$, and Remark~\ref{rem:commutative-general} gives the
first assertion.

If, moreover, $A$ is a \C*, the above argument shows that $T\bigl(\{a\}^{cc}\bigr)$ is a subalgebra of $B$ for every
element $a\in A_{sa}$. Now Proposition~\ref{prop:comm-subalg} completes the proof.
\end{proof}

Applying the above general principles in various concrete situations enables us to obtain new incidences in which
every unital surjective spectral isometry is a Jordan isomorphism.

The next two results were obtained in~\cite{MMRud} under the additional assumption that the spectrum
of the domain algebra is totally disconnected.

\begin{theorem}\label{thm:typeI-hausdorff-prim}
Let\/ $T\colon A\to B$ be a unital surjective spectral isometry from a unital type~\textup{I} \C* with Hausdorff primitive ideal space
onto a unital \C*~$B$. Then\/ $T$ is a Jordan isomorphism.
\end{theorem}
\begin{proof}
By replacing $T$ with $T^{-1}$, we can interchange the roles of $A$
and~$B$. If $\prim B$ is Hausdorff, every primitive ideal is maximal and thus every irreducible representation $\pi$ maps onto a
simple unital \C*. As we assume $B$ to be of type~I, every irreducible image $\pi(B)$ contains the compact operators and consequently,
is a finite-dimensional simple unital \C*, that is, of the form $M_n(\CC)$. Every unital spectral isometry on $M_n(\CC)$
is multiplicative or anti-multiplicative (\cite[Proposition~2]{Aup93}; for an alternative argument, see
\cite[Example~5.4]{Mat09}). Therefore, the hypotheses of Proposition~\ref{prop:gen-criteria} are satisfied
and thus $T$ is a Jordan isomorphism.
\end{proof}
\begin{theorem}\label{thm:typeI-hausdorff-sep}
Let\/ $T\colon A\to B$ be a unital surjective spectral isometry from a separable unital \C* with Hausdorff spectrum
onto a unital \C*~$B$. Then\/ $T$ is a Jordan isomorphism.
\end{theorem}
\begin{proof}
If the spectrum of $A$, that is, the space of all equivalence classes of irreducible representations, is Hausdorff then it is homeomorphic
to $\prima$; hence all irreducible images are unital and simple. Since they are separable too, Glimm's theorem
\cite[Theorem~6.8.7]{Ped} entails that they are of the form $M_n(\CC)$ (i.e., $A$ is of type~I); compare \cite[Corollary~12]{MMRud}.
Now the proof is completed as above.
\end{proof}

The class of type~I \C*s with Hausdorff spectrum is quite large, it encompasses in particular all
continuous-trace \C*s. For a more detailed discussion, see~\cite{MMRud}.
On the other hand, our methods also apply to very different non-simple \C*s which are far from type~I.

\begin{theorem}\label{thm:main}
Let\/ $A$ and\/ $B$ be unital \C*s, and let\/ $T\colon A\to B$ be a unital surjective spectral isometry.
Suppose that\/ $A$ has real rank zero and no tracial states and that\/ $\prima$ contains a dense subsets of closed points.
Then\/ $T$ is a Jordan isomorphism.
\end{theorem}
\begin{proof}
As before we swap the roles of $A$ and $B$ to obtain the additional assumptions on~$B$.
By hypothesis, there exists a separating family $\idI$ of maximal ideals of~$B$. Let $I\in\idI$.
Then $B/I$ has real rank zero and no tracial states (and, of course, is unital simple).
By \cite[Theorem~3.1]{MM04b}, every unital spectral isometry from $B/I$ onto itself is a Jordan isomorphism
and it is well known that Jordan isomorphisms of simple algebras are either multiplicative or anti-multiplicative.
Consequently, we can apply Proposition~\ref{prop:gen-criteria} to obtain the result.
\end{proof}

\begin{rem}\label{rem:glimm-procedure}
The reduction theory employed in this paper is closely related to the reduction via Glimm ideals used in~\cite{MaSou},
\cite{MMRud} and~\cite{Co09}. For a unital \C* $A$, $\prima$ is Hausdorff if and only if every Glimm ideal of $A$ is maximal
\cite[Lemma~9]{MMRud}. Moreover, the condition that $T\bigl(\{a\}^{cc}\bigr)$ is an algebra for every selfadjoint $a \in A$
implies that $T$ is \textit{$Z$-multiplicative}, i.e., $T(xz)=(Tx)(Tz)$ for all $x\in A$ and all $z\in Z(A)$. This in turn entails
invariance of Glimm ideals which is the essential tool in the Glimm reduction procedure.
\end{rem}

\end{document}